\documentclass[a4paper,12pt]{article}
\usepackage[T1]{fontenc}
\usepackage[utf8]{inputenc}
\usepackage{amsfonts} 
\usepackage{graphics}
\newtheorem{theorem}{Theorem}[section]
\newtheorem{lemma}{Lemma}[section]
\newtheorem{definition}{Definition}[section]
\newtheorem{proposition}[theorem]{Proposition}

\def\dint{\displaystyle \int}

\newcommand{\proof}{{\bf Proof:} }

\newcommand{\cqd}{ \hfill $\Box$ }
\begin{document}
\begin{center}
{\Large Necessary Conditions for Optimal Control of SPDE with locally monotone coefficients} 
\\
\end{center} 

\vspace{0.3cm}

\begin{center}
{\large  Edson A. Coayla-Teran} \\
\textit{Universidade Federal da Bahia-UFBA\\
Av. Ademar de Barros s/n, Instituto de Matem\' atica, Salvador, BA, Brasil, CEP 40170-110, e-mail: coayla@ufba.br\\
Phone: 0051-071-32836279, Fax: 0051-071-32836276}
\end{center}
\begin{abstract}
 The aim of this paper is to derive a maximum principle  for a control problem governed by a stochastic partial differential equation (SPDE) with locally monotone coefficients. In particular, necessary conditions for optimality  for this stochastic optimal control problem are obtained by using the adjoint backward stochastic partial differential equation (BSPDE).  
\end{abstract}

\section{Introduction} Let $T>0$ some fixed time, $\mathcal{O}$ and $H$ two real separable Hilbert spaces and $V$ a reflexive Banach space with $H'$ and $V'$ the dual spaces of $H$ and $V$ respectively. Consider the following initial value problem involving a controlled SPDE: 
\begin{equation}\label{coeq1}
du(t)=(A(t)u(t)+B(u(t),\Phi(t)))dt+\Xi(u(t))dW(t),\,  u(0)=u_0.
\end{equation}
where $B:V\times\mathcal{O}\times\Omega\rightarrow V'$ and $\Xi: H\times\Omega\rightarrow L_2(U;H)$ are measurable maps and the map $\Phi:[0,T]\times\Omega\rightarrow\mathcal{O}$  is progressively measurable which will play the role of the control, $B$ satisfies a locally monotone condition (see condition {\bf A2} below) $A:[0,T]\times V\times\Omega\rightarrow V'$ is  progressively measurable such that $A(t,\cdot,\omega)$  with $(t,\omega)\in[0,T]\times\Omega$ is a linear operator, $\left\{W(t)\right\}_{t\geq0}$ be a cylindrical Wiener process. For more details about the assumptions on the coefficients see the Section 2.\\
\indent The typical optimal control problem is the minimization of  $\mathcal{J}(\Phi)$ where $\mathcal{J}$ is a cost functional (in our case it is given by (\ref{cost})) with $\Phi$ belonging to  $\mathcal{U},$ the set $\mathcal{U}$ will denote the  admissible controls associated with the controlled initial value problem (\ref{coeq1}). The present work deals with deriving the maximum principle (or  Pontryagin's maximum principle) for that control problem, more specifically, we will provide necessary conditions for optimality for that optimal control problem. To attain our objective we use the adjoint backward stochastic partial differential equation (\ref{BSE}) associated with (\ref{coeq1}). This argument is well known (see  \cite{AH3} for example), combining this with some ideias from \cite{AH2} and \cite{HL} we get necessary conditions for optimality.\\
\indent It is important to remark that the use of  backward stochastic differential equations (BSDEs) for deriving the maximum principle for forward controlled stochastic equations was first discussed by Bismut \cite{B}. Several authors showed the relation between BSDEs and the maximum principle for stochastic differential equations, see \cite{AH2}, \cite{AH3}, \cite{O}, \cite{P}, and \cite{YZ} for example.\\
\indent It is important to mention that the works \cite{AH2} , \cite{AH3} and  \cite{FT} use BSDEs to derive the maximum principle in infinite dimension but the results of these papers cannot be applied in the study of the equation in (\ref{coeq1}) because they assume that the derivative of the map $B(\cdot,\cdot)$ in  (\ref{coeq1}) with respect to the variable $u$ and with respect to the variable $\Phi$ is uniformly bounded (which implies that the map $B(\cdot,\cdot)$ is Lipschitz continuous), whereas in the present work we assume that the map $B(\cdot,\cdot)$ is locally monotone and its derivative satisfies {\bf A8} and {\bf A9}.\\
\indent Another important question in the control theory is to study the existence of optimal control, we do not include it in this paper. This question was studied in \cite{CMF} where the existence of optimal control was demonstrated under the assumption that the map $B(\cdot,\cdot)$ satisfies a locally monotone condition and additional conditions.\\ 
\indent The present article is organized in the following way: in Section 2, we present the basic spaces, the norms, properties and notations which we are going to work with in the subsequent sections. In Section 3, we formulate the optimal control problem and provide estimates to linear equation associated with the equation (\ref{coeq1}). The Section 4 is devoted to study the Adjoint Equation, we demonstrate the existence and uniqueness of a solution and some useful estimates which we will use in the next section.The Section 5 is dedicated to state the goal of the present work: Necessary  Conditions for Optimal Control.
\section{Notation}
Let $H$ be a real separable Hilbert space.  Let $V$ be a 
reflexive Banach space. Identify $H$ with its dual $H'$ and denote the dual of $V$ by $V'$. Let
$$
V\subset H\cong H'\subset V'
$$
where the inclusions are assumed to be dense and compact. The triad $(H,V,V')$ is 
known as a \textit{Gelfand triple}. We will denote by $\|\cdot\|_V,$ $\|\cdot\|,$ $\|\cdot\|_{V'}$ the norms in $V,$ $H,$ and $V'$ respectively. The inner product in $H$ and the duality scalar product between $V$ and $V'$ will be denoted by $(\cdot,\cdot)$ and $\langle\cdot,\cdot\rangle$ respectively. Furthermore we will denote by $\|\cdot\|_{\mathcal{O}}$ the norm in $\mathcal{O}$ and by $L_2(U;H)$  the space of all Hilbert-Schmidt operators from $U$ to $H$ with inner product  $\langle\cdot,\cdot\rangle_2$ and norm $\|\cdot\|_2$, for $U$ a real separable Hilbert space.\\
\indent Let $(\Omega,\mathcal{F},\mathbb{P})$ be a complete probability space. Let $\left\{W_t\right\}_{t\geq0}$ be a $R-$cylindrical Wiener process on $U$ with its completed natural filtration $\mathcal{F}_t,$ with $t\geq0.$\\
\indent To simplify notation we use the letter $\mathbb{T}$ for the interval $[0,T].$ $\mathbb{E}(X)$ denotes the mathematical expectation of the random variable $X.$ We abbreviate ``\textit{almost surely} $\omega\in\Omega.$" to  a.s.\\
\indent Let $B$ be a Banach space with norm $\|\cdot\|_B$ and let $\mathcal{B}(B)$ denote the Borel $\sigma-$algebra of $B$. The space $L^2(\Omega\times\mathbb{T};B)$ is the set of all $\mathcal{F}\otimes\mathcal{B}(\mathbb{T})-$measurable processes $u:\Omega\times\mathbb{T}\rightarrow B$ which are $\mathcal{F}_t-$ adapted and $\mathbb{E}(\int_{\mathbb{T}}\|u\|^2_Bdt)<\infty.$ The constant $c_{HV}$ is such that $\|v\|^2\leq c_{HV}\|v\|^2_V$ for all $v\in V$.\\
\indent In order to get solutions to (\ref{coeq1}), we state the following conditions on the coefficients:
Suppose there exist positive constants  $\theta,$ $K$ and a positive adapted process $f\in L^1(\Omega\times\mathbb{T};\mathbb{R})$ such that the following conditions hold for all $v,$ $v_1,$ $v_2\in V,$ $\Phi\in\mathcal{O}$ and a.e. $(t,\omega)\in\mathbb{T}\times\Omega$.
\begin{enumerate}
\item[{\bf A1})]  (Hemicontinuity) The map $s\rightarrow\langle B(v_1+sv_2,\Phi),v\rangle+\langle A(v_1+sv_2,v\rangle$ is continuous on $\mathbb{R}.$
\item[{\bf A2})] (Local monotonicity)
$$\begin{array}{rl}
2\langle B(v_1,\Phi)-B(v_2,\Phi),v_1-v_2\rangle&+\|\Xi(v_1)-\Xi(v_2)\|_2^2\leq\\
&(K+\rho(v_2))\|v_1-v_2\|^2,
\end{array}
$$
where $\rho: V\rightarrow[0,+\infty)$ is a mensurable function and locally bounded in $V.$
\item[{\bf A3})] (Coercivity) 
$$
2\langle A(t)v,v\rangle\leq-\theta\|v\|_V^2
$$
and
$$ 
2\langle B(v,\Psi),v\rangle+\|\Xi(v)\|_2^2\leq K\|v\|^2+f(t).
$$
\item[{\bf A4})] (Growth) 
$$
\|A(t)v\|^2_{V'}+\|B(v,\Psi)\|^2_{V'}+\leq (f(t)+K\|v||^2_V)(1+\|v\|^2).
$$
\end{enumerate}
In this work, we understand that the stochastic process $u_{\Phi}$ is a solution to 
the problem in (\ref{coeq1}) in the following sense.
\begin{definition} Let $u_0$ be a  random variable which does not depend on $W(t).$ The 
stochastic process $(u_{\Phi}(t))_{t\in\mathbb{T}}\in L^2(\Omega\times\mathbb{T};V),$ $\mathcal{F}_t-$ adapted, with a.s. sample paths continuous in $H$, is a solution to (\ref{coeq1}) if it satisfies the equation:
\begin{equation}\label{iska}\begin{array}{rl}
(u_{\Phi}(t),v)\!=&(u_0,v)+\!\dint_0^t\left\langle B(u_{\Phi}(s),\Phi(s)),v\right\rangle ds+\dint_0^t(A(s)u_{\Phi}(s)),v)ds+\\
&+\dint_0^t(v,\Xi(u_{\Phi}(s))dW(s))
\end{array}
\end{equation}
a.s. for all $v\in V$ and $t\in\mathbb{T}.$
\end{definition}
\indent {\bf Uniqueness} means indistinguishability.\\
We need the following existence of solutions theorem, the result is a particular case of \cite[Th. ~1.1]{LR} .
\begin{theorem}\label{exsth} Let $u_0\in L^4(\Omega, V)$. Suppose that {\bf A1}-{\bf A4} is satisfied and there is a constant $C$ such that
\begin{equation}\begin{array}{rl}
\|\Xi(v)\|^2_2\leq& C(f(t)+\|v\|^2), \ \ \ v\in V;\\
\rho(v)\leq &C(1+\|v\|_V^2)(1+\|v\|^2)\ \ \ v\in V.
\end{array}
\end{equation}
The problem (\ref{iska}) has a unique solution $u_{\Phi}$ which  has a.s. sample paths continuous in $H$ and satisfies
\begin{equation}\label{estiska}
\mathbb{E}\left[\sup_{t\in\mathbb{T}}\|u_{\Phi}(t)\|^2+\int_0^T\|u_{\Phi}(t)\|^2_Vdt\|\right]\leq C,
\end{equation} 
where $C$ is a constant dependent of $T$ and $u_0.$
\end{theorem}
\begin{proof} The proof follows from theorem 1.1 of \cite{LR} because our assumptions on coefficients of the equation in (\ref{coeq1}) satisfies the conditions of that theorem.  
\end{proof}\cqd
\section{Formulation of the Control Problem and Estimates}
Let $\mathcal{O}_1$ be a fixed convex subset of $\mathcal{O}.$ The map $\Phi\in L^2(\Omega\times\mathbb{T};\mathcal{O})$ such that $\Phi(t)\in \mathcal{O}_1$ a.e., a.s. is called \textit{admissible control}. The set of admissible controls will be denoted by $\mathcal{U}.$
Let us now define the {\it cost functional}:
\begin{equation}\label{cost}
\mathcal{J}(\Phi):=\mathbb{E}\left[\dint_0^T\Bigl(\mathcal{L}(u_{\Phi}(s),\Phi(s))\Bigr)ds+(\mathcal{K}(u_{\Phi}(T)))\right],\ \Phi\in\mathcal{U}
\end{equation}
whenever the integral in (\ref{cost}) exists and is finite. We will assume that the mappings $\mathcal{L}:H\times\mathcal{O}\rightarrow\mathbb{R}_+$ and $\mathcal{K}:H\rightarrow\mathbb{R}_+$ are measurable and satisfy the following conditions:
\begin{enumerate}
\item[{\bf H1})] we assume that $\mathcal{L},$ $\mathcal{K}$ satisfies
$$\begin{array}{rl}|\mathcal{L}(u_1,\Phi_1)-\mathcal{L}(u_2,\Phi_2)|&\leq C_{\mathcal{L}}(\|u_1-u_2\|^2+\|\Phi_1-\Phi_2\|^2_{\mathcal{O}}),\\
|\mathcal{K}(u_1)-&\mathcal{K}(u_2)|\leq C_{\mathcal{K}}\|u_1-u_2\|^2,\end{array}$$
for all  $u_1,\ \  u_2\in H,$ $\Phi_1,\ \ \Phi_2\in\mathcal{O}$ and $C_{\mathcal{L}},$ $C_{\mathcal{K}}$ are positive constants.
\item[{\bf H}2)]  the mappings $\mathcal{L}(\cdot,\cdot)$ and $\mathcal{K}(\cdot)$ are Fr\'echet differentiable and their derivatives $\mathcal{L}_u(\cdot,\cdot),$ $\mathcal{L}_\Phi(\cdot,\cdot)$ and $\mathcal{K}'(\cdot)$ satisfies
$$\begin{array}{rl}
\|\mathcal{L}_u(u,\Phi)\|&\leq k_{\mathcal{L}}\|u\|,\ \mathcal{L}_u(0,\Phi)=0,\ \|\mathcal{L}_\Phi(u,\Phi)\|\leq k_{\mathcal{L}}\|\Phi\|_{\mathcal{O}} \textmd{ and}\\
\|\mathcal{K}'(u)\|&\leq k_{\mathcal{K}}\|u\|
\end{array}
$$
for all $\Phi\in\mathcal{O},$ $u\in H$ and $k_{\mathcal{L}},$ $k_{\mathcal{K}}$ positive constants.
\end{enumerate}
\indent We will assume the following additional conditions for the equation (\ref{coeq1}) which are useful to demonstrate the existence and uniqueness of the solution to linear equation (\ref{linequ}). Let $u,\ v,\ v_1,\ v_2$ be in $V$ and $\Phi,\,\Phi_1\in\mathcal{O}$, suppose that there are positive constants $\theta_1$ $K_2,$ $K_3$, $K_4$, $K_5$ and $K_6$ such that:
\begin{enumerate}
\item[{\bf A5})] 
The mapping $\Xi(\cdot)$ is Fr\'echet differentiable and its derivative $\Xi'(u)\in L^2(H,L_2(U;H)).$
\item[{\bf A6})] $B:V\times\mathcal{O}\rightarrow V',$ is a mapping such that the map $B(\cdot,\Phi)$ is Fr\'echet differentiable and we denote by $B_u(u,\Phi)$ its Fr\'echet derivative at the point $u.$
\item[{\bf A7})] For  $u\in V$ fixed, we have 
$$\begin{array}{rl}
\langle B_u(u,\Phi)(v),v_1-v_2\rangle\leq K_2\|v_1-v_2\|_V\|v\|\rho_1(u),
\end{array}
$$
where $\rho_1: H\rightarrow[0,+\infty)$ is a measureable function and locally bounded in $H.$
\item[{\bf A8})] 
$
\langle (B_u(u,\Phi)-B_u(v,\Phi_1))(w),z\rangle \leq K_3\|z\|_V\rho_1(u-v)\|w\|
$
\item[{\bf A9})] 
$
2\langle B_u(v_1,\Phi)(v),v\rangle+2\|\Xi_u(v_1)(v)\|_2^2\leq\theta_1(\rho_1(v_1))^2\|v\|^2.
$
\item[{\bf A10})]
$
\|A(t)(v)\|_{V'}^2+\|B_u(u,\Phi)(v)\|^2_{V'}+\|\Xi_u(v)\|^2_{V'}\leq (1+K_4\|v\|^2_V)(1+\|u\|^2).
$
\item[{\bf A11})] $B:V\times\mathcal{O} \times\Omega\rightarrow V',$ is a mapping such for each $u\in V$ the mapping $\Phi\in \mathcal{O}\rightarrow B(u,\Phi)\in V'$ is Fr\'echet differentiable and we denote by $B_{\Phi}(u,\Phi)$ its Fr\'echet derivative at the point $\Phi$ which  satisfies:
$$
\|B_\Phi(u,\Phi)(\Upsilon)-B_\Phi(v,\Phi)(\Upsilon_1)\|_{V'}\leq K_5\|\Upsilon-\Upsilon_1\|_{\mathcal{O}}
$$
for all  $\Upsilon,\ \ \Upsilon_1$ $\in\mathcal{O}.$ 
\item[{\bf A12})]  $\|\Xi'(v)(w)-\Xi'(u)(w)\|_2 \leq  K_6\|w\|,$\ \  $\Xi'(0)=0.$
\end{enumerate}
\indent Our control problem is to minimize (\ref{cost}) over $\mathcal{U}.$ When $\Phi^{\star}\in\mathcal{U}$ satisfies
\begin{equation}\label{minizacost}
\mathcal{J}(\Phi^{\star})=\inf_{\Phi}\mathcal{J}(\Phi)
\end{equation}
is called an \textit{optimal control}. The corresponding solution $u_{\Phi^{\star}}$ of (\ref{coeq1}) and the pair $(u^{\star},\Phi^{\star})$ are called respectively an \textit{optimal solution} and \textit{optimal pair} of the optimal problem control (\ref{coeq1}) and (\ref{minizacost}).\\
\indent Let $\Phi^{\star}$ be  an optimal control,  the corresponding solution  $u_{\Phi^{\star}}$ of (\ref{coeq1}) will be denoted by $u^{\star}:=u_{\Phi^{\star}}$. Now we will consider the initial value problem involving a linear SPDE:
\begin{equation}\begin{array}{ll}\label{linequ}
dP(t)=&A(t)P(t)dt+(B_u(u^{\star}(t),\Phi^{\star})(P(t)))dt+\\
&+(B_{\Phi}(u^{\star}(t),\Phi^{\star})(\Phi(t)))dt+\Xi'(u^{\star}(t))(P(t))dW(t),\\
P(0)&=0.
\end{array}
\end{equation}
\indent Our first objective is to demonstrate the existence and uniqueness of the solution to (\ref{linequ}) (see Theorem \ref{TEUC}). We need to prove this fact because we have not found in the literature a result that can be applied to obtain existence and uniqueness of the solution of the equation (\ref{linequ}) whose coefficients satisfy the conditions {\bf A5} -{\bf A12}. To demonstrate that fact we use well known arguments and some ideas from \cite{HL}. With this result in hand we can demonstrate important estimates for that solution which will be used to obtain the maximum principle in Sect. 5.\\
\indent  Previous definitions are necessary, let  $\{u_{\Phi}(t)\}_ {t\in\mathbb{T}}$ be a $V$-valued stochastic process such that $\sup_{t\in\mathbb{T}}\|u_{\Phi}(t)\|^2<\infty$ and $\int_0^T\|u_{\Phi}(t)\|^2_Vdt<\infty$ for a.e.\\
\indent For each $M,$ a nonnegative integer, we define the following stopping times:
$$
\mathcal{T}_M^{u_{\Phi}}:=\left\{\begin{array}{lr}
T,\textmd{ if }\sup_{t\in\mathbb{T}}\|u_{\Phi}(t))\|^2<M&\\
\inf\left\{t\in\mathbb{T}:\sup_{t\in\mathbb{T}}\|u_{\Phi}(t)\|^2\geq M\right\},\textmd{otherwise}
\end{array}
\right.
$$
and
$$
\mathcal{S}_M^{u_{\Phi}}:=\left\{\begin{array}{lr}
T,\textmd{ if }\int_0^T\|u_{\Phi}(t))\|^2_Vdt<M&\\
\inf\left\{t\in\mathbb{T}:\int^t_0\|u_{\Phi}(s)\|^2_Vds\geq M\right\},\textmd{otherwise.}
\end{array}
\right.
$$
We define $\mathcal{T}_M:=\min\{\mathcal{T}_M^{u_{\Phi}},\mathcal{S}_M^{u_{\Phi}}\}$ and  $u^M(t):=u(t\wedge\mathcal{T}_M).$\\ 
\indent From the properties of the stopping time it holds that $\lim_{M\rightarrow\infty}\mathcal{T}_M=T$  (see the appendix of \cite{HL}).\\
\indent For  $M,$  integer nonnegative we consider the following initial value problem 
\begin{equation}\label{linstop}\begin{array}{rl}
d P^{M}(t)=&A(t)P^M(t)dt+B_u((u^\star)^M(t),\Phi^\star(t))(P^{M}(t))dt+\\
+&B_{\Phi}((u^{\star})^M(t),\Phi^\star(t))\Phi(t)dt+\Xi'((u^{\star})^M(t))(P^{M}(t))dW(t),\\
P^M(0)&=0.
\end{array}
\end{equation}
\indent In order to demonstrate the existence of the solution of (\ref{linequ}), first we will demonstrate the existence of the solution of (\ref{linstop}). With this end we will use the Faedo-Galerkin's method. Let $\{w_k\}_{k=1}^{\infty}$ be a complete orthonormal basis for $H$ and orthogonal in $V.$ For each $n\in \mathbb{Z}^+,$ we consider $H_n:=span\left\{w_1,w_2,\ldots,w_n\right\}$ equipped with the norm induced from $H.$ We define by $\Pi_n: H\rightarrow H_n$ the orthogonal projection such that, $\Pi_nh:=\sum_{i=1}^n(h,w_i)w_i$ for $h\in H,$ we can extend $\Pi$ to a projection operator $\Pi'_n:V'\rightarrow V_n'$ by $\Pi '_nv:=\sum_{i=1}^n\left\langle v,w_i\right\rangle w_i$ with $V_n=H_n=V_n'.$ Analogously if $\left\{v_i\right\}_{i=1}^{\infty}$ be the set of eigenfunctions of $R, $ $K_n:=span\left\{v_1,v_2,\dots,v_n\right\}$ we may to define the projection operator $P_n$ from $K$ into $K_n$ such that $P_nh:=\sum_{i=1}^n(h,v_i).$ We define the following truncations:
\begin{equation}\begin{array}{rl}
(u_{\Phi})_n:=\Pi_n u_{\Phi},&( B_u)_n(u,\Phi):=\Pi_nB_u(u,\Phi),\\
 (B_{\Phi})_n(u,\Phi):=&\Pi_n B_{\Phi}(u,\Phi),\, (\Xi')_n(u):=\Pi_n\Xi'(u),\\
 A_n(t)(u):=&\Pi'_nA(t)(u)\, W^n:=P_nW
\end{array}
\end{equation}
for $t\in\mathbb{T}$, $u\in V$.\\
\indent Analogously  we can assume that it is possible to define the truncation $\Phi_n$ for $\Phi\in\mathcal{O}.$\\
\indent Now, we will consider the following finite dimensional stochastic differential equation:
\begin{equation}\label{linaprostop}\begin{array}{rl}
dP^M_n(t)=&\left(A_n(t)P^M_n(t)+(B_{u})_n(((u^\star)^M)_n(t),\Phi^\star(t))(P^M_n(t))\right)dt+\\
&(B_{\Phi})_n(((u^\star)^M)_n(t),\Phi^\star(t))(\Phi_n(t))dt) +\\
&(\Xi')_n(((u^\star)^M)_n(t))(P^M_n(t))dW^n(t),
\end{array}
\end{equation}
a.s., $t\in\mathbb{T}$ with initial condition $P^M_n(0)=0.$
\begin{theorem}\label{TEUC} Suppose that $\rho_1(u)\leq\|u\|$ and conditions {\bf A1}-{\bf A12} are satisfied, then, there exists a unique 
solution $(Z(t))_{t\in\mathbb{T}}$ to the equation (\ref{linequ}) which is a $V-$valued,
$\mathcal{F}\times\mathcal{B}_{\mathbb{T}}-$measurable process and adapted to the filtration 
$(\mathcal{F}_t)_{t\in\mathbb{T}}.$ Furthermore, there exists a constant $c$ such that
\begin{equation}\label{estimatlinequ}\begin{array}{rl}
\mathbb{E}(\exp\left(-\int_0^T\rho_1(u^\star(t))dt\right)&\|P(T)\|^2)+\\
\mathbb{E}\int_0^T&(\exp\left(-\rho_1(u^\star(s))ds\right)\|P(t)\|^2_Vdt\leq\\
c\mathbb{E}(\int_0^T&\|u^\star(t)\|^2dt+\mathbb{E}\int_0^T\|\Xi'(u^\star(t)\|^2_2dt),\\
\mathbb{E}(\left(\exp-\int_0^T\rho_1(u^\star(t))dt\right)^2&\|P(T)\|^2)+\\
\mathbb{E}(\int_0^T&(\exp\left(-\rho_1(u^\star(s))ds\right)\|P(t)\|^2_Vdt)^2\leq\\
c(\mathbb{E}(\int_0^T&\|u^\star(t)\|^2dt)^2 +\mathbb{E}(\int_0^T\|\Xi'(u^\star(t)\|^2_2dt))^2
\end{array}
\end{equation}
\end{theorem}
\begin{proof}We can apply the theory of finite dimensional It\^o equation (see \cite{KR1} Sect. 3.3 ) to the equation (\ref{linaprostop}) to demonstrate that there exists a unique solution $P^M_n,$ which has almost surely continuous trajectories in $H.$\\ \indent Now we need estimates to the equation (\ref{linaprostop}). Using the It\^o formula we get
$$\begin{array}{rl}
\|P^M_n(t)\|^2=&2\int_0^t\left\langle A_n(t)P^M_n(s), P^M_n(s)\right\rangle ds+\\
&2\int_0^t\left\langle (B_u)_n(((u^\star)^M)_n(s),\Phi^\star(s))(P^M_n(s)),P^M_n(s)\right\rangle ds+\\
&2\int_0^t\left\langle (B_{\Phi})_n(((u^\star)^M)_n(s),\Phi^\star(t))(\Phi_n(s),P^M_n(s)\right\rangle ds+\\
&\int_0^t\|(\Xi_u)_n(((u^\star)^M)_n(s))(P^M_n(s))\|^2_2 ds+\\
&2\int_0^t(P^M_n(s),(\Xi')_n(((u^\star)^M)_n(s))(P^M_n(t))dW^n(s)).
\end{array}
$$
Using the equation above, the It\^o formula with $e(t)=\exp(-\int_0^t(\rho_1(((u^{\star})^M))_n)^2(s)ds)$ and from our assumptions \textbf{A10} , \textbf{A11} and  \textbf{A12} we have
$$\begin{array}{rl}
\mathbb{E}(e(t)\|P^M_n(t)\|^2)&+\frac{\theta}{2}\mathbb{E}(\int_0^t e(s)\|P^M_n(t)\|^2_Vdt)\leq c [\mathbb{E}(\int_0^te'(s)\|P^M_n(s)\|^2ds)+\\
&\mathbb{E}\int_0^te(s)(\rho_1((u^\star)^M)_n(s))^2\|P^M_n(s)\|^2ds+\\
&\mathbb{E}\int_0^te(s)\|\Phi(s)\|^2_{\mathcal{O}}ds+K_6\mathbb{E}(\int_0^te(s)\|P^M_n(s)\|^2ds)]
\end{array}
$$  
Where $c$ is  a constant independent of $n$ and $M$. The last inequality and the Gronwall's inequality imply that 
\begin{equation}
\mathbb{E}(e(t)\|P^M_n(t)\|^2)+\frac{\theta}{2}\mathbb{E}(\int_0^T e(t)\|P^M_n(t)\|^2_Vdt)\leq c\mathbb{E}(\int_0^T\|\Phi(s)\|^2_{\mathcal{O}}ds),
\end{equation}
where $c$ is  a constant independent of $n.$\\
From the inequality above and using the assumption about $\rho_1$ we obtain
\begin{equation}\begin{array}{rl}\label{estimatlinear}
\mathbb{E}(\int_0^t\|P^M_n(t)\|^2_Vdt)&\leq\\
\mathbb{E}(&(e(T))^{-1}\int_0^te(s)\|P^M_n(s)\|^2_Vds)\leq\\
c\mathbb{E}(&\exp( \int_0^T\|u^\star(t\wedge\mathcal{T}_M)(t)\|^2)dt)\mathbb{E}(\int_0^T\|\Phi(s)\|^2_{\mathcal{O}}ds)\leq\\
c\mathbb{E}(&\int_0^T\|\Phi(s)\|^2_{\mathcal{O}}ds),
\end{array}
\end{equation}
where $c$ is a constant that depends on $M.$\\
Analogously we can prove that
\begin{equation}
\mathbb{E}(\|P^M_n(t)\|^2)^2\leq c
\end{equation}
and
\begin{equation}
\mathbb{E}(\|P^M_n(t)\|^4)+\mathbb{E(}\int_0^T \|P^M_n(t)\|^2_Vdt)^2\leq c
\end{equation}
where $c$ is a constant independent of $n$ but dependent of $M,$ $T$ and $\Phi$.\\
Thus, the inequality (\ref{estimatlinear}) implies that the sequence $(P_n^M)$ is bounded in the space $L^2(\Omega\times\mathbb{T};V).$ Then, there exists a subsequence of $(P_n^M)$ relabeled the same and $P^M \in L^2(\Omega\times\mathbb{T};V)$ such that letting $n\rightarrow\infty$ we have
\begin{equation}\label{weakconv2}
(P_n^M)\rightharpoonup P^M.
\end{equation} 
\indent Now, we will prove that
\begin{equation}\label{convT}
(B_u)_n (((u^\star)^M)_n,\Phi^\star)(P^M_n)\rightharpoonup B_u((u^\star)^M,\Phi^\star)(P^M)
\end{equation}
in $L^2(\Omega\times\mathbb{T};V').$ Let $v\in V$ and $v_n=\Pi_n v.$ To demonstrate (\ref{convT}), first we observe that 
\begin{equation}\label{convTI}\begin{array}{rl}
\langle(B_u)_n (((u^\star)^M)_n,\Phi^\star)(P^M_n),v\rangle&=\langle B_u(((u^\star)^M)_n,\Phi^\star)(P^M_n),v_n\rangle=\\
\langle(B_u(((u^\star)^M)_n,\Phi^\star)(P^M_n),v_n\rangle&-\langle(B_u(((u^\star)^M)_n,\Phi^\star)(P^M_n),v\rangle+\\
\langle B_u(((u^\star)^M)_n,\Phi^\star)(P^M_n),v\rangle&-\langle B_u((u^\star)^M,\Phi^\star)(P^M_n),v\rangle+\\
\langle B_u((u^\star)^M,\Phi^\star)(P^M_n),v\rangle&-\langle B_u((u^\star)^M,\Phi^\star)(P^M),v\rangle+\\
&\langle B_u((u^\star)^M,\Phi^\star)(P^M),v\rangle.
\end{array}
\end{equation}
 We need to estimate each term in the last inequality. From \textbf{A7} and the assumption about $\rho_1$, we have
 \begin{equation}\label{convT1}\begin{array}{rl}
 \mathbb{E}(\int_0^T|\langle(B_u(((u^\star)^M)_n(t)&,\Phi^\star(t))(P^M_n(t)),v_n\rangle-\\
 \langle B_u(((u^\star)^M)_n(t),&\Phi^\star(t))(P^M_n(t)),v\rangle|dt)\leq\\
  c\|v_n-v\|_V(\mathbb{E}&(\int_0^T\|(P^M)_n)\|^2dt)^{1/2}(\mathbb{E}(\int_0^T\|(u^\star)^M_n\|^2dt))^{1/2},
 \end{array}
 \end{equation}
where $c$ is a generic constant.\\
Using \textbf{A8}  and the assumption about $\rho_1$, we have
 \begin{equation}\label{convT2}\begin{array}{rl}
 \mathbb{E}(\int_0^T\langle B_u(((u^\star)^M)_n(t),&\Phi^\star(t))(P^M_n(t)),v\rangle-\\
 &\langle B_u((u^\star)^M(t),\Phi^\star(t))(P^M_n(t)),v\rangle dt)\leq\\
 &c\|v\|_V(\mathbb{E}(\int_0^T\|P^M)_n(t)dt\|^2)^{1/2}\times\\
 &(\mathbb{E}(\int_0^T\|((u^\star)^M)_n(t)-(u^\star)^M(t)\|^2dt))^{1/2},
\end{array} 
 \end{equation}
 where $c$ is a generic constant.\\
 And from (\ref{weakconv2}) we have
 \begin{equation}\label{convT3}
  B_u((u^\star)^M,\Phi^\star)(P^M_n)\rightharpoonup B_u((u^\star)^M,\Phi^\star)(P^M).
 \end{equation}
Using (\ref{convT1}), (\ref{convT2}), (\ref{convT3})  in (\ref{convTI}) we obtain
$$\begin{array}{rl}
\lim_{n\rightarrow\infty}\int_0^T&\langle(B_u)_n (((u^\star)^M)_n(t),\Phi^\star(t))(P^M_n)(t),\lambda(t)\rangle dt=\\
&\int_0^T\langle B_u ((u^\star)^M(t),\Phi^\star(t))P^M(t),\lambda(t)\rangle dt
\end{array}
$$
for all $\lambda=v\phi,$ with $v\in V,$ $\phi$ a $V-$valued process, $\mathcal{F}\times\mathcal{B}(\mathbb{T})-$measurable, $\mathcal{F}_t-$adapted and for a.e. $(\omega,t)$ bounded. As the maps $\lambda$ with these properties are dense in $L^2(\Omega\times\mathbb{T};V)$ we get (\ref{convT}).\\
Using \textbf{A11} it is possible to demonstrate that
\begin{equation}\label{weakconv3}
B_\Phi((u^\star)^M_n,\Phi^\star)(\Phi_n)\rightharpoonup B_\Phi((u^\star)^M,\Phi^\star)(\Phi)
\end{equation} 
\indent Using (\ref{convT}), (\ref{weakconv3}) and taking the limit $n\rightarrow\infty$ in  (\ref{linaprostop}) we obtain
 \begin {equation}\begin{array}{rl}
 (P^M(t),v)=&\dint_0^t(A(s),P^M(s))ds+\int_0^t(B_u((u^{\star})^M(s),\Phi^{\star}(s))(P^M(s)),v)ds+\\
 &\dint^t_0(B_{\Phi}((u^{\star})^M(s),\Phi^{\star}(s))(\Phi(s)),v)ds+\\
 &\dint_0^t(v,\Xi'(u^{\star}(s)(P^M(s)))dW(s)).
 \end{array}
 \end{equation}
 From this and arguing as \cite{HL} it is possible to conclude that there is a process $\{P^M_{\Phi}(t)\}_{t\in\mathbb{T}}$ in $L^2(\Omega\times\mathbb{T};V)$ which has a.s. continuous trajectories in $H$ and satisfies  (\ref{linstop}) a.s for all $v\in V,$ $t\in\mathbb{T}.$ Using well known methods we can demonstrate that that process is the unique solution.\\ 
 \indent We proceed to demonstrate the existence of the solution of  (\ref{linequ}), we will use the argument of \cite{HL}. Let $\Omega_M=\{\omega\in\Omega: P^M_{\Phi}(\omega,\cdot)\textmd{ satisfies }(\ref{linstop})\textmd{ for all }t\in\mathbb{T},\, v\in V\textmd{ and has continuous trajectories in }H\},$  $\Omega'=\bigcap_{M=1}^{\infty}\Omega_M,$  and $S=\bigcup_{K=1}^{\infty}\bigcup_{1\leq M\leq K}\{\omega\in\Omega':\mathcal{T}_M=T\textmd{ and there is }t\in \mathbb{T}, P^M_{\Phi}(\omega,t)\neq P^K_{\Phi}(\omega,t)\},$ it is possible to demonstrate that $\mathbb{P}(S)=0.$ \\
 \indent We define $\Omega''=\bigcup_{M=1}^{\infty}\{\mathcal{T}_M=T\},$ then $\mathbb{P}(\Omega'\bigcap\Omega'')\setminus S)=1.$ For $\omega\in\Omega'\bigcap\Omega'')\setminus S$ there is a natural number $M_0$ such that $\mathcal{T}_M=T$ for all $M\geq M_0.$ Thus, $(u^{\star})^M(t)= (u^{\star})^{M_0}(t)$ for all $t\in\mathbb{T}$ and for all $M\geq M_0.$ From (\ref{linstop}) we have
  \begin {equation}\label{lineqend}\begin{array}{rl}
 (P^M(t),v)=&\dint_0^t(A(s),P^M(s))ds+\int_0^t(B_u(u^{\star}(s),\Phi^{\star}(s))(P^M(s)),v)ds+\\
 &\dint^t_0(B_{\Phi}(u^{\star}(s),\Phi^{\star}(s))(\Phi(s)),v)ds+\\
 &\dint_0^t(v,\Xi'(u^{\star}(s)(P^M(s)))dW(s)),
 \end{array}
 \end{equation}
 for all $M\geq M_0,$ and all $t\in\mathbb{T},$ $v\in V.$ Then,
 $$
 \lim_{M\rightarrow\infty}\int_0^T\|P^M(t)-P^M_0(t)\|_V^2dt=0
 $$
 and
 $$
 \lim_{M\rightarrow\infty}\|P^M(t)-P^M_0(t)\|^2dt=0
 $$
 for all  $t\in\mathbb{T}.$\\
 \indent For each $t\in\mathbb{T}$ we define 
 $$
 P(\omega,t):=P^M_0(\omega,t)\lim_{M\rightarrow\infty}P^M(\omega,t).
 $$
 Thanks to (\ref{lineqend}) we obtain
  \begin {equation}\label{lineqend1}\begin{array}{rl}
 (P(t),v)=&\dint_0^t(A(s),P(s))ds+\int_0^t(B_u(u^{\star}(s),\Phi^{\star}(s))(P(s)),v)ds+\\
 &\dint^t_0(B_{\Phi}(u^{\star}(s),\Phi^{\star}(s))(\Phi(s)),v)ds+\\
 &\dint_0^t(v,\Xi'(u^{\star}(s)(P(s)))dW(s)),
 \end{array}
 \end{equation}
 for all $\omega\in\Omega'\bigcap\Omega'')\setminus S),$ $t\in\mathbb{T},$ $v\in V.$ This demonstrates the existence of the solution of (\ref{linequ}). The inequalities in (\ref{estimatlinequ}) and unicity are demonstrated as  in \cite{HL}.
\end{proof}\cqd
\section{The Adjoint Equation}
It is known that there is a connection between BSDEs and the maximum principle for stochastic differential equations. Thus, it seems natural to consider that there is the same connection with the equation (\ref{coeq1}). With this objective, let us first consider the \textit{Hamiltonian}:
\begin{equation}\label{hamilt}
\begin{array}{rl}
\mathcal{H}:&V\times\mathcal{O}\times V\times L_2(U;H)\rightarrow\mathbb{R},\\
\mathcal{H}(u,\Phi,v,Z)&:=\mathcal{L}(u,\Phi)+\langle B(u,\Phi),v\rangle+\langle\Xi(u),Z\rangle_2
\end{array}
\end{equation}
Now, we consider the adjoint equation associated with (\ref{coeq1}) which is given by the following terminal value problem involving a BSPDE:
\begin{equation}\label{BSE}\begin{array}{rl}
-dv^{\Phi}(t)=&(A^{\star}v^{\Phi}(t)+\nabla_u\mathcal{H}(u^{\Phi}(t),\Phi(t),v^{\Phi}(t),Z^{\Phi}(t)))dt-Z^{\Phi}(t)dW(t),\\
v^{\Phi}(T)=&\nabla\mathcal{K}(u^{\Phi}(T)).
\end{array}
\end{equation} 
Where $\nabla$ denotes the gradient.\\
According to (\ref{hamilt}) we can rewrite (\ref{BSE}) in the following way:
\begin{equation}\label{BSER}
\begin{array}{rl}
dv^{\Phi}(t)=&(-A^{\star}v^{\Phi}(t)-\nabla_u\mathcal{L}(u^{\Phi}(t),\Phi(t))-\nabla_u\langle B(u^{\Phi}(t),\Phi(t),v^{\Phi}(t))\rangle+\\
-&\nabla_u\langle\Xi(u^{\Phi}(t)),Z^{\Phi}(t)\rangle_2)dt+Z^{\Phi}(t)dW(t),\ \ 0\leq t\leq T,\\
v^{\Phi}(T)=&\nabla\mathcal{K}(u^{\Phi}(T)).
\end{array}
\end{equation}
\begin{definition}In this work we understand that the $\mathcal{F}_t-$adapted pair $(v,Z)$ in $L^2(\Omega\times\mathbb{T};H)\times L^2(\Omega\times\mathbb{T};L_2(U;H))$ is a solution of (\ref{BSER}) if the following is satisfied
\begin{equation}\label{BSEI}\begin{array}{rl}
(v^{\Phi}(t),y)=&(v^{\Phi}(T),y)+\int^T_t(\langle A^{\star}v^{\Phi}(s),y\rangle+(\nabla_u\mathcal{L}(u^{\Phi}(s),\Phi(s)),y)_Vds+\\
&\int^T_t\langle B_u(u^{\Phi}(s),\Phi(t))(y),v^{\Phi}(s)\rangle+(\Xi_u(u^{\Phi}(s))(y),Z^{\Phi}(s))_2)ds-\\
&\int^T_t(y,Z^{\Phi}(s)dW(s))
\end{array}
\end{equation}
a.s. for all $y\in V$ and $t\in\mathbb{T}.$
\end{definition}
Let $M>0$ be a  natural number, given $\Phi\in\mathcal{U},$ let $u^{\Phi}$ be the corresponding solution to SPDE in (\ref{coeq1}) and $u^{\Phi,M}:=(u^{\Phi})^M.$\\
\indent In order to prove that there is a solution (\ref{BSE}) we will consider the following terminal value problem:
\begin{equation}\label{BSERstp}\begin{array}{rl}
dv^{\Phi, M}(t)&=(-A^{\star}v^{\Phi,M}(t)-\nabla_u\mathcal{L}({u^{\Phi,M}(t)},\Phi(t))-\\
&\nabla_u\langle B(u^{\Phi,M}(t),\Phi(t)),v^{\Phi,M}(t)\rangle-\nabla_u\langle\Xi(u^{\Phi,M}(t)),Z^{\Phi,M}(t)\rangle_2)dt+\\
&Z^{\Phi,M}(t)dW(t),\\
v^{\Phi,M}(T)=&\nabla\mathcal{K}(u^{\Phi}(T)).
\end{array}
\end{equation}
\indent To demonstrate that there is a solution to the equation (\ref{BSERstp}) we will use the Faedo-Galerkin's method.  Thus, for each $n\in\mathbb{Z}^+,$ we will consider the projected equation corresponding with (\ref{BSERstp}):
\begin{equation}\label{BSERn}\begin{array}{rl}
d(v^{\Phi,M})_n(t)=&(-A^{\star}_n(v^{\Phi,M})_n(t)-(\nabla_u\mathcal{L})_n(((u^{\Phi,M}))_n(t),\Phi_n(t))-\\
&(\nabla_u)_n\langle B_n(((u^{\Phi,M}))_n(t),\Phi_n(t),(v^{\Phi})_n(t))\rangle+\\
&(\nabla_u)_n\langle\Xi_n(((u^{\Phi,M}))_n(t)),(Z^{\Phi})_n(t)\rangle_2)dt+\\
&(Z^{\Phi})_n(t)dW^n(t),\\
(v^{\Phi})_n(T)=&(\nabla\mathcal{K})_n((u^{\Phi})(T)).
\end{array}
\end{equation}
\begin{proposition}  
Suppose that {\bf H1} - {\bf H2}  and {\bf A8} - {\bf A12} is satisfied. There is a solution $(v^{\Phi,M},Z^{\Phi,M})$ to (\ref{BSERstp}) for each $M=1,2,\ldots$ and
\begin{equation}\label{estimequM}
\mathbb{E}(\|v^{\Phi,M)}\|^2(t)+\mathbb{E}(\int^T_t\|v^{\Phi,M}\|_V^2ds)+\mathbb{E}(\int^T_t\|Z^{\Phi,M}\|^2_2ds)\leq c\ \ 
\end{equation}
with $t\in\mathbb{T}$, to a constant $c$ dependent of $M.$ 
\end{proposition}
\proof It is possible to demonstrate that the coefficients of the equation (\ref{BSERn}) satisfy the conditions of the Theorem 3.4 of   \cite{BD}. Thus, let $((v^{\Phi,M})_n,(Z^{\Phi,M})_n)$ be a solution  to the equation (\ref{BSERn}).\\
\indent Now we need to obtain estimates to the solution $((v^{\Phi,M})_n,(Z^{\Phi,M})_n).$ To do this we employ the It\^o formula to that solution and to the function $e^{r(t)}$, with $r(t)=(\int_0^t1+2K_6+\frac{K_2^2}{\theta}((\rho_1((u^{\Phi,M})^n(s))^2))ds$ and taking expectation we get
$$\begin{array}{rl}
\mathbb{E}(e^{r(t)}&\|(v^{\Phi,M})_n(t)\|^2)+\mathbb{E}(\int_t^Te^{r(s)}\|(Z^{\Phi,M})_n(s)\|^2_2ds)=\mathbb{E}(e^{r(T)}\|(v^{\Phi,M})_n(T)\|)+\\
&2\mathbb{E}(\int_t^Te^{r(s)}\langle A^{\star}  (v^{\Phi,M})_n(s), (v^{\Phi,M})_n(s)\rangle ds)+\\
&2\mathbb{E}(\int_t^Te^{r(s)}((\nabla_u)_n\mathcal{L}((u^{\Phi})_n(s),\Phi_n(s)),(v^{\Phi,M})_n(s))ds)+\\
&2\mathbb{E}(\int_t^Te^{r(s)}((\nabla_u)_n\langle B((u^{\Phi,M})_n(s),\Phi_n(s),(v^{\Phi,M})_n(s))\rangle,(v^{\Phi,M})_n(s))ds)+\\
&2\mathbb{E}(\int_t^Te^{r(s)}((\nabla_u)_n(\Xi((u^{\Phi})_n(s)),(Z^{\Phi,M})_n(s))_2 ,(v^{\Phi,M})_n(s))ds)-\\
&\mathbb{E}(\int_t^Tr'(s)e^{r(s)}\|(v^{\Phi,M})_n(s)\|^2ds).
\end{array}
$$
From this and using  {\bf A8}-{\bf A12} we have
$$
\begin{array}{rl}
\mathbb{E}(e^{r(t)}&\|(v^{\Phi,M})_n(t)\|^2)+\frac{1}{2}\mathbb{E}(\int_t^Te^{r(s)}\|(Z^{\Phi,M})_n(s)\|^2_2ds)+\\
&\frac{\theta}{2}\mathbb{E}(\int_t^Te^{r(s)}\|(v^{\Phi,M})_n(s)\|_V^2ds)\leq \mathbb{E}(e^{r(T)}\|v^{\Phi,M}(T)\|)+\\
&\mathbb{E}(\int_t^Te^{r(s)}K_{\mathcal{L}}\|(u^{\Phi,M})_n(s)\|^2ds)+\mathbb{E}(\int_t^Te^{r(s)}\|(v^{\Phi,M})_n(s)\|^2ds)+\\
&\frac{K_2^2}{\theta}\mathbb{E}(\int_t^Te^{r(s)}(\rho_1((u^{\Phi})_n(s)))^2\|(v^{\Phi,M})_n(s)\|^2ds)+\\
&2\mathbb{E}(\int_t^Te^{r(s)}K_6\|(v^{\Phi,M})_n(s)\|^2ds)-\\
&\mathbb{E}(\int_t^Tr'(s)e^{r(s)}\|(v^{\Phi,M})_n(s)\|^2ds).
\end{array}
$$
Since that $r'(t)=1+2K_6+\frac{K^2_2}{\theta}(\rho_1((u^{\Phi,M})^n(t)))^2$ from the last inequality we have
$$
\begin{array}{rl}
\mathbb{E}(e^{r(t)}&\|v^{\Phi,M}(t)\|^2)+\frac{1}{2}\mathbb{E}(\int_t^Te^{r(s)}\|Z^{\Phi,M}(s)\|^2_2ds)+\\
&\frac{\theta}{2}\mathbb{E}(\int_t^Te^{r(s)}\|v^{\Phi,M}(s)\|_V^2ds\leq\mathbb{E}e^{r(T)}\|v^{\Phi,M}(T)\|.\\
\end{array}
$$
Then,
 $$
 \mathbb{E}(e^{r(t)}\|(v^{\Phi,M})_n(t)\|^2)+\mathbb{E}(\int_t^Te^{r(s)}\|(v^{\Phi,M})_n(s)\|_V^2ds+\mathbb{E}(\int_t^Te^{r(s)}\|(Z^{\Phi,M})_n(s)\|^2_2ds) \leq c
 $$
 where $c$ is a constant independent of $n$.\\
 Using the same argument as the demonstration of Theorem \ref{TEUC}, we can demonstrate that
 $$
 \mathbb{E}\|(v^{\Phi,M})_n(t)\|^2+\mathbb{E}(\int_t^T\|(v^{\Phi,M})_n(s)\|_V^2ds+\mathbb{E}(\int_t^T\|(Z^{\Phi,M})_n(s)\|^2_2ds) \leq c
 $$
 for a constant $c$ which depends on $M$ but independent of $n.$\\
\indent Since the sequences $((v^{\Phi,M})_n)$ and  $((Z^{\Phi,M})_n)$ are bounded in the space $L^2(\Omega\times\mathbb{T};V)$ and $L^2(\Omega\times\mathbb{T}:L_2(U;H))$ respectively, there exists subsequences of $((v^{\Phi,M})_n)$ and $((Z^{\Phi,M})_n)$ relabeled the same and $v^M \in L^2(\Omega\times\mathbb{T};V)$ and $Z^M\in L^2(\Omega\times\mathbb{T}:L_2(U;H))$ such that letting $n\rightarrow\infty$ we have
\begin{equation}\label{weakconv4}
\begin{array}{rl}
((v^{\Phi,M})_n)\rightharpoonup v^M,\\
((Z^{\Phi,M})_n)\rightharpoonup Z^M.
\end{array}
\end{equation}
Following the same argument as the proof of the Theorem \ref{TEUC}, we can demonstrate the following weak convergence $(\nabla_u\mathcal{L})_n(((u^{\Phi,M}))_n,\Phi_n)\rightharpoonup (\nabla_u\mathcal{L})((u^{\Phi,M}),\Phi),$ $(\nabla_u)_n\langle B(((u^{\Phi,M}))_n,\Phi_n),(v^{\Phi,M})_n\rangle \rightharpoonup (\nabla_u)\langle B(u^{\Phi,M},\Phi),(v^{\Phi,M})\rangle,$ $(\nabla_u)_n\langle\Xi(((u^{\Phi,M}))_n),(Z^{\Phi,M})_n\rangle_2\rightharpoonup(\nabla_u)\langle\Xi((u^{\Phi,M}),(Z^{\Phi,M})\rangle_2$ as $n\rightarrow\infty.$\\
Letting $n\rightarrow\infty$ in (\ref{BSERn}) we obtain a solution $(v^{\Phi,M},Z^{\Phi,M})$ to equation (\ref{BSERstp}).\cqd
\begin{theorem}\label{BSExUn} Given $u\in V,$ and $\Phi\in\mathcal{O}.$ Under the assumptions of the last proposition, there is a unique solution $(v^{\Phi},Z^{\Phi})$ to (\ref{BSE}). 
\end{theorem}
\proof
With the goal to demonstrate that there is a solution to equation (\ref{BSE},) we will use the same argument as in the proof of Theorem \ref{TEUC}, let $\Omega_M=\{\omega\in\Omega:(v^{\Phi,M}(\omega,\cdot),Z^{\Phi,M}(\omega,\cdot))\textmd{ for all }t\in\mathbb{T},v\in V \}.$ Set $\Omega'=\bigcap_{M=1}^{\infty}\Omega_M$ and $S=\bigcup_{K=1}^{\infty}\bigcup_{1\leq M\leq K}\{\omega\in\Omega':\mathcal{T}_M=T\textmd{ and there is }t\in\mathbb{T},\  (v^{\Phi,M}(\omega,t),Z^{\Phi,M}(\omega,t)) \neq(v^{\Phi,K}(\omega,t),Z^{\Phi,K}(\omega,t))\},$ it is possible to demonstrate that $\mathbb{P}(S)=0.$\\
\indent Let $\Omega"=\bigcup_{M}\{\mathcal{T}_M=T\},$ it is possible to demonstrate that $\mathbb{P}(((\Omega'\cap\Omega"))\setminus S)=1.$ Let  $\omega\in(\Omega'\cap\Omega"))\setminus S,$ for this $\omega$ there is a natural number $M_0$ such that 
$\mathcal{T}_M=T$ for all $M\geq M_0$ implying that $u^{\Phi,M}(s)=u^{\Phi}(s)$ for all $s\in\mathbb{T}$ and for all $M\geq M_0.$ The equation (\ref{BSERstp}) implies
\begin{equation}\label{BSEIntSt}\begin{array}{rl}
(v^{\Phi,M}(t),y)\!\!&=(v^{\Phi,M}(T),y)+\int^T_t(\langle A^{\star}v^{\Phi,M}(s),y\rangle+(\nabla_u\mathcal{L}(u^{\Phi}(s),\Phi(s)),y)_Vds\\
\!\!\!\!&+\int^T_t\langle B_x(u^{\Phi}(s),\Phi(t))(y),v^{\Phi,M}(s)\rangle+(\Xi_x(u^{\Phi}(s))(y),Z^{\Phi,M}(s))_2)ds\\
&-\int^T_t(y,Z^{\Phi}(s)dW(s))
\end{array}
\end{equation}
a.s. for all $y\in V$ and $t\in\mathbb{T}.$
Hence
$$\begin{array}{rl}
\|(v^{\Phi,M}-v^{\Phi,M_0})(t)\|^2&=\int_0^T\|(v^{\Phi,M}-v^{\Phi,M_0})(s)\|^2_Vds=\\
&\int_0^T\|(Z^{\Phi,M}-Z^{\Phi,M_0})(s)\|^2_2ds=0
\end{array}
$$
for all $t\in\mathbb{T}.$ Thus, taking $v^{\Phi}(\omega,t)=v^{\Phi,M_0}(\omega,t)$ and $Z^{\Phi}(\omega,t)=Z^{\Phi,M_0}(\omega,t)$ for $t\in\mathbb{T}$ and $\omega\notin S,$ we obtain the solution $(v^{\Phi},Z^{\Phi})$ to equation  (\ref{BSE}).
\cqd
\section{A stochastic maximum principle}
Let $\Phi^{\star}$ be an optimal control and let $u^{\Phi^{\star}}$ be the corresponding solution of (\ref{coeq1}). This solution will be denoted briefly by $u^{\star}$. Let $\Phi$ be such that $\Phi^{\star}+\Phi\in\mathcal{U}.$ For a given $0\leq\epsilon\leq1$ consider the control:\\
$$
\Phi_{\epsilon}(t)=\Phi^{\star}(t)+\epsilon\Phi(t),\ \ \ t\in\mathbb{T},
$$
to which will be associated the  solution of (\ref{coeq1}), $u^{\Phi_{\epsilon}},$ which will be denoted by $u_{\epsilon}$ for short.\\
Now, we will obtain some important estimates that we will use to obtain the maximum principle.
\begin{lemma}\label{estpm} With the notation above we have
\begin{equation}\label{estimat3}
\sup_{t\in\mathbb{T}}\mathbb{E}(e(t)(\|u_{\epsilon}-u^{\star})(t)\|^2)=O(\epsilon^2),
\end{equation}
where $e(t)=e^{-\int_0^tK+\rho(u^{\star}(s))ds}$ and $K$ is the constant from {\bf A2}.
\end{lemma}
 \begin{proof} Using the It\^o's formula and taking the expectation,  we get
 $$\begin{array}{rl}
 \mathbb{E}(e(t)\|\!\!\!\!&(u_{\epsilon}-u^{\star})(t)\|^2)\leq\\
 &2\mathbb{E}(\int_0^te(s)\langle B(u_{\epsilon}(s),\Phi_{\epsilon}(s))-B(u^{\star}(s),\Phi^{\star}(s)),u_{\epsilon}-u^{\star}(t)\rangle ds))+\\
&2\mathbb{E}(\int_0^te(s)\langle A(s)(u_{\epsilon}-u^{\star})(s),(u_{\epsilon}-u^{\star})(s)\rangle ds)+\\
&\mathbb{E}(\int_0^te(s)\|\Xi(u_{\epsilon}(s),\Phi_{\epsilon}(s))-\Xi(u^{\star}(s),\Phi^{\star}(s))\|^2_2ds)+\\
&2\mathbb{E}(\int_0^te'(s))\|(u_{\epsilon}-u^{\star})(t)\|^2ds)\leq\\
&2\mathbb{E}(\int_0^te(s)\langle B(u_{\epsilon}(s)),\Phi_{\epsilon}(s))-B(u_{\epsilon}(s),\Phi^{\star}(s)),(u_{\epsilon}-u^{\star})(s)\rangle ds))+\\
&2\mathbb{E}(\int_0^te(s)\langle B(u_{\epsilon}(s),\Phi^{\star}(s))-B(u^{\star}(s)),(\Phi^{\star}(s),(u_{\epsilon}-u^{\star})(s)\rangle ds))+\\
&\mathbb{E}(\int_0^te'(s)\|(u_{\epsilon}-u^{\star})(s)\|^2ds)+\mathbb{E}(\int_0^te(s)\|\Xi(u_{\epsilon}(s))-\Xi(u^{\star}(s))\|^2_2ds)+\\
&2\mathbb{E}(\int_0^t\langle A(s)(u_{\epsilon}-u^{\star})(s),(u_{\epsilon}-u^{\star})(s)\rangle ds)
\end{array}
 $$
Since {\bf A11}, we have 
$$\begin{array}{rl}
\mathbb{E}(\int_0^t\|B(u_{\epsilon}(s))\!\!\!\!&,\Phi_{\epsilon}(s))-B(u_{\epsilon}(s),\Phi^{\star}(s))\|^2_{V'}ds)=\\
\mathbb{E}(\int_0^t\|\int_0^1\!\!\!\!&B_{\Phi}(u_{\epsilon}(s),\Phi^{\star}(s)+r(\Phi_{\epsilon}(s)-\Phi^{\star}(s)))(\Phi_{\epsilon}(s)-\Phi^{\star}(s))dr\|^2_{V'}ds)\leq\\
&K_5\epsilon^2\mathbb{E}(\int_o^t\|\Phi(t)\|^2_{\mathcal{O}}dt).
\end{array}
$$ 
Combining these two inequalities above,  {\bf A2} and {\bf A3},  we have
$$
 \mathbb{E}(e(t)\|(u_{\epsilon}-u^{\star})(t)\|^2)+\frac{\theta}{2}\mathbb{E}(\int_0^te(s)\|(u_{\epsilon}-u^{\star})(s)\|^2_Vds)\leq\epsilon^2C,\\
 $$
 where $C$ represents a generic constant. Thus, (\ref{estimat3}) follows from the last inequality.
\end{proof}
 \cqd\\
 \indent Let $0\leq r\leq1$ and $0<\epsilon<1.$ Set $e_1(t):=e^{-\int_0^t(\rho_1(u^{\star}(s)+r(u_{\epsilon}(s)-u^{\star}(s))))^2ds}$ we need to prove the property
\begin{equation} \label{specP}
 \mathbb{E}(e_1^{-4}(T)<\infty.
\end{equation}
It is possible to demonstrate that under special assumptions  this property is satisfied. To prove it we use a similar argument as in \cite{HL}, subsection  2.5. 
 \begin{lemma} With the notation above, let $K$ and $\theta$ the constants appearing in {\bf A3}. Suppose that $\gamma:=\sup_{u\in H}\|\Xi(u)\|^2_2<\infty$ and $(K-2\theta)^2-48\gamma c_{HV}^2>0.$ Then, the property (\ref{specP}) is satisified.
 \end{lemma}
 \begin{proof} The proof is similar as in the proof of the Theorem 2.5.1 in \cite{HL}.
 \end{proof}\cqd\\
 Now we will return to estimates to obtain the maximum principle.  
 \begin{lemma}\label{estimatLemma} Assume the conditions of the last lemma and {\bf A5}-{\bf A12}.  Let $\Delta_{\epsilon}(t)=\frac{u_{\epsilon}(t)-u^{\star}(t)-\epsilon P(t)}{\epsilon}$ where $P$ is the solution to the equation (\ref{linequ}). Then,
\begin{equation}\label{estimat4}
\lim_{\epsilon\rightarrow0}\sup_{t\in\mathbb{T}}\mathbb{E}(\|\Delta_{\epsilon}(t)\|^2)=0
\end{equation}
 \end{lemma}
 \begin{proof} According to the notation we have
 $$\begin{array}{rl}
 \Delta_{\epsilon}(t)&=\int_0^tA(s)\Delta_{\epsilon}(s)ds+\\
 &\int_0^t\frac{1}{\epsilon}(B(u_{\epsilon}(s),\Phi_{\epsilon}(s))-B(u^{\star}(s),\Phi^{\star}(s)))-B_u(u^{\star}(s),\Phi^{\star}(s))P(s)ds+\\
 &\int_0^t\frac{1}{\epsilon}(\Xi(u_{\epsilon}(s))-\Xi(u^{\star}(s)))-\Xi_u(u^{\star}(s))P(s))dW(s)-\\
 &\int_0^t B_{\Phi}(u^{\star}(s),\Phi^{\star}(s))\Phi(s)ds=\int_0^tA(s)\Delta_{\epsilon}(s)ds+\\
 &\int_0^t\frac{1}{\epsilon}(B(u_{\epsilon}(s),\Phi_{\epsilon}(s))-B(u^{\star}(s),\Phi_{\epsilon}(s)))-B_u(u^{\star}(s),\Phi^{\star}(s))P(s)ds+\\
 &\int_0^t\frac{1}{\epsilon}(B(u^{\star}(s),\Phi_{\epsilon}(s))-B(u^{\star}(s),\Phi^{\star}(s)))-B_{\Phi}(u^{\star}(s),\Phi^{\star}(s))\Phi(s)ds+\\
 &\int_0^t\frac{1}{\epsilon}(\Xi(u_{\epsilon}(s))-\Xi(u^{\star}(s)))-\Xi_u(u^{\star}(s))P(s))dW(s),
 \end{array}
 $$
 for $t\in\mathbb{T}.$ Using the It\^o's formula with $e_1(t)=e^{-\int_0^t(\rho_1(u^{\star}(s)+r(u_{\epsilon}(s)-u^{\star}(s))))^2ds}$ for $r\in[0,1]$and taking the expectation, we have
 $$\begin{array}{rl}
 \mathbb{E}(e_1(t)\|\Delta_{\epsilon}(t)\|^2)+\theta\mathbb{E}(\int_0^te_1(s)\|\Delta_{\epsilon}(s)\|^2_Vds)&\leq\mathbb{I}+\mathbb{II}+\mathbb{III}+\\
 &\mathbb{E}(\int_0^te_1'(s)\|\Delta_{\epsilon}(s)\|^2ds),
 \end{array}
 $$
 where 
 $$\begin{array}{rl}
 \mathbb{I}=&2\mathbb{E}(\int_0^te_1(s)\langle \frac{1}{\epsilon}(B(u_{\epsilon}(s),\Phi_{\epsilon}(s))-\\
 &B(u^{\star}(s),\Phi_{\epsilon}(s)))-B_u(u^{\star}(s),\Phi^{\star}(s))P(s), \Delta_{\epsilon}(s)\rangle)ds),\\
 \mathbb{II}=&2\mathbb{E}(\int_0^te_1(s)\langle \frac{1}{\epsilon}(B(u^{\star}(s),\Phi_{\epsilon}(s))-\\
 &B(u^{\star}(s),\Phi^{\star}(s)))-B_{\Phi}(u^{\star}(s),\Phi^{\star}(s))\Phi(s),\Delta_{\epsilon}(s)\rangle ds)\\
 \mathbb{III}=&\mathbb{E}(\int_0^te_1(s)\|\frac{1}{\epsilon}(\Xi(u_{\epsilon}(s))-\Xi(u^{\star}(s)))-\Xi_u(u^{\star}(s))P(s)\|^2_2ds)
 \end{array}
 $$
We need to estimate each term above. Using our assumption about $\rho_1$ we get  
$$\begin{array}{rl}
\mathbb{I}&=2\mathbb{E}(\int_0^te_1(s)\langle \frac{1}{\epsilon}\int_0^1B_u(u^{\star}(s)+r(u_{\epsilon}(s)-u^{\star}(s)),\Phi_{\epsilon}(s))(u_{\epsilon}(s)-u^{\star}(s))dr, \\
&\Delta_{\epsilon}(s)\rangle ds)+2\mathbb{E}(\int_0^te_1(s)\langle-B_u(u^{\star}(s),\Phi^{\star}(s))P(s), \Delta_{\epsilon}(s)\rangle ds)=\\
&2\mathbb{E}(\int_0^te_1(s)\langle\int_0^1B_u(u^{\star}(s)+r(u_{\epsilon}(s)-u^{\star}(s)),\Phi_{\epsilon}(s))(\Delta_{\epsilon}(s)), \Delta_{\epsilon}(s)\rangle ds)+\\
&2\mathbb{E}(\int_0^te_1(s)\langle\int_0^1 B_u(u^{\star}(s)+r(u_{\epsilon}(s)-u^{\star}(s)),\Phi_{\epsilon}(s))P(s)- B_u(u^{\star}(s),\Phi^{\star}(s))P(s),\\
&\Delta_{\epsilon}(s)\rangle drds)=\\
&2\mathbb{E}(\int_0^te_1(s)\langle\int_0^1B_u(u^{\star}(s)+r(u_{\epsilon}(s)-u^{\star}(s)),\Phi_{\epsilon}(s))(\Delta_{\epsilon}(s)),\Delta_{\epsilon}(s)\rangle drds)+\\
&2\mathbb{E}(\int_0^te_1(s)\langle \int_0^1B_u(u^{\star}(s)+r(u_{\epsilon}(s)-u^{\star}(s)),\Phi_{\epsilon}(s))P(s)-B_u(u^{\star}(s),\Phi^{\star}(s))P(s),\\
& \Delta_{\epsilon}(s)\rangle drds)\leq\\
&\frac{\theta}{4}\mathbb{E}(\int_0^te_1(s)\|\Delta_{\epsilon}(s)\|^2_Vds)+\frac{8K_2^2}{\theta}\mathbb{E}(\int_0^te_1(s)(\rho_1(u^{\star}(s)+r(u_{\epsilon}(s)-u^{\star}(s))))^2\times\\
&\|\Delta_{\epsilon}(s)\|^2ds)+\\
&\frac{8}{\theta}(\mathbb{E}(\sup_{t\in\mathbb{T}}\|P(t)\|^4))^{1/2(}\mathbb{E}(\int_0^te_1(s)\|(u_{\epsilon}-u^{\star})(s)\|^2ds)^2)^{1/2}.
\end{array}
$$
\indent About the second term
$$\begin{array}{rl}
\mathbb{II}&=2\mathbb{E}(\int_0^te_1(s)\langle\int_0^1(B_{\Phi}(u^{\star}(s),\Phi^{\star}+r(\Phi_{\epsilon}(s)-\Phi^{\star}(s)))-\\
&B_{\Phi}(u^{\star}(s),\Phi^{\star}(s))(\Phi(s)), \Delta_{\epsilon}(s)\rangle drds)\leq\\
&\frac{4}{\theta}\mathbb{E}(\int_0^te_1(s)\int_0^1\|B_{\Phi}(u^{\star}(s),\Phi^{\star}+r(\Phi_{\epsilon}(s)-\Phi^{\star}(s)))-\\
&B_{\Phi}(u^{\star}(s),\Phi^{\star}(s))(\Phi(s))\|^2_{V'}drds)+\\
&\frac{\theta}{4}\mathbb{E}(\int_0^te_1(s)\|\Delta_{\epsilon}(s)\|^2_Vds).
\end{array} 
$$
Now we will estimate the third term
$$\begin{array}{rl}
\mathbb{III}&\leq\mathbb{E}(\int_0^te_1(s)\|\int_0^1\Xi_u(u^{\star}(s)+r(u_{\epsilon}(s)-u^{\star}(s)))(\Delta_{\epsilon}(s))dr\|^2_2ds)+\\
&\mathbb{E}(\int_0^te_1(s)\|\int_0^1\Xi_u(u^{\star}(s)+r(u_{\epsilon}(s)-u^{\star}(s)))(P(s))-\\
&\Xi_u(u^{\star}(s))(P(s))dr\|^2_2ds)\leq\\
&\mathbb{E}(\int_0^te_1(s)\|\int_0^1\Xi_u(u^{\star}(s)+r(u_{\epsilon}(s)-u^{\star}(s)))(P(s))-\\
&\Xi_u(u^{\star}(s))(P(s))dr\|^2_2ds)+\\
&K_6\mathbb{E}(\int_0^te_1(s)\|\Delta_{\epsilon}(s)\|^2ds).
\end{array}
$$
Thus,
\begin{equation}\label{estimafinal1}\begin{array}{rl}
 \mathbb{E}(e_1(t)&\|\Delta_{\epsilon}(t)\|^2)+\frac{\theta}{2}\mathbb{E}(\int_0^te_1(s)\|\Delta_{\epsilon}(s)\|^2_Vds)\leq\\
 &\frac{4}{\theta}\mathbb{E}(\int_0^te_1(s)\int_0^1\|B_{\Phi}(u^{\star}(s),\Phi^{\star}+r(\Phi_{\epsilon}(s)-\\
 &\Phi^{\star}(s)))-B_{\Phi}(u^{\star}(s),\Phi^{\star}(s))(\Phi(s))\|^2_{V'}drds)+\\
 &\frac{8}{\theta}(\mathbb{E}(\sup_{t\in\mathbb{T}}\|P(t)\|^4))^{1/2}(\mathbb{E}(\int_0^te_1(s)\|u_{\epsilon}(s)-u^{\star}(s)\|^2ds)^2)^{1/2}+\\
&\mathbb{E}(\int_0^te_1(s)\|\int_0^1\Xi_u(u^{\star}(s)+r(u_{\epsilon}(s)-u^{\star}(s)))(P(s))-\\
&\Xi_u(u^{\star}(s))(P(s))dr\|^2_2ds)+K_6\mathbb{E}(\int_0^te_1(s)\|\Delta_{\epsilon}(s)\|^2ds).
 \end{array}
 \end{equation}
\indent Using the Lemma \ref{estpm}, it is possible to demonstrate that $\mathbb{E}(\int_0^te_1(s)\|u_{\epsilon}(s)-u^{\star}(s)\|^2ds)^2=O(\epsilon^2).$ Also, using the assumptions {\bf A11}-{\bf A12} and dominated convergence it is possible to demonstrate that $\mathbb{E}(\int_0^t\|\int_0^1\Xi_u(u^{\star}(s)+r(u_{\epsilon}(s)-u^{\star}(s)))(P(s))-\Xi_u(u^{\star}(s))(P(s))dr\|^2_2ds\rightarrow0$ and $\mathbb{E}(\int_0^te_1(s)\int_0^1\|B_{\Phi}(u^{\star}(s),\Phi^{\star}+r(\Phi_{\epsilon}(s)-\Phi^{\star}(s)))-B_{\Phi}(u^{\star}(s),\Phi^{\star}(s))(\Phi(s))\|^2_{V'}drds)\rightarrow0$ as $\epsilon\rightarrow0^+.$ Combining these facts in the inequality (\ref{estimafinal1}) and Gronwall's inequality we obtain
\begin{equation}\label{estimatfinal2}
\sup_{t\in\mathbb{T}}\mathbb{E}(e_1(t)\|\Delta_{\epsilon}(t)\|^2)\rightarrow0,\ \textmd{as } \epsilon\rightarrow0^+.
\end{equation}
Analogously we can demonstrate
\begin{equation}\label{estimatfinall3}
\mathbb{E}((e_1(t))^2\|\Delta_{\epsilon}(t)\|^4)\rightarrow0,\ \textmd{as } \epsilon\rightarrow0^+.
\end{equation}
In order to demonstrate (\ref{estimat4}) we observe that
$$\mathbb{E}(\|\Delta_{\epsilon}(t)\|^2)\leq(\mathbb{E}(e_1(t)\|\Delta_{\epsilon}(t)\|^2 ))^{1/2}(\mathbb{E}(e_1^{-4}(t)))^{1/4}(\mathbb{E}(e_1^2(t)\|\Delta_{\epsilon}(t)\|^4 ))^{1/4}$$
Hence, using (\ref{specP}), (\ref{estimatfinal2}) and (\ref{estimatfinall3}) we obtain (\ref{estimat4}).
\end{proof}\cqd
\begin{theorem} We assume {\bf A5}-{\bf A12}. Then, for each $\epsilon>0,$
\begin{equation}\label{estPrinMax}\begin{array}{rl}
\mathcal{J}(\Phi_{\epsilon})-\mathcal{J}(\Phi^{\star})&=
\epsilon\mathbb{E}(\mathcal{K}'(u^{\star}(T),P(T)))+\\
&\epsilon\mathbb{E}(\int_0^T(\mathcal{L}_u(u^{\star}(t),\Phi^{\star}(t)), P(t))dt)+\\
&\mathbb{E}(\int_0^T(\mathcal{L}(u^{\star}(t),\Phi_{\epsilon}(t))-\mathcal{L}(u^{\star}(t),\Phi^{\star}(t)))dt)+o(\epsilon).
\end{array}
\end{equation}
\end{theorem}
\proof Since
\begin{equation}\label{estimfunct}\begin{array}{rl}
\mathcal{J}(\Phi_{\epsilon})-\mathcal{J}(\Phi^{\star})=&\mathbb{E}(\mathcal{K}(u_{\epsilon})(T)-\mathcal{K}(u^{\star})(T))+\\
\mathbb{E}(\int_0^T\mathcal{L}(u_{\epsilon}(s),\Phi_{\epsilon}(s))-\mathcal{L}(u^{\star}(s),\Phi^{\star}(s)ds).
\end{array}
\end{equation}
We need to calculate each term above. Since
$$
\mathcal{K}(u_{\epsilon}(T))-\mathcal{K}(u^{\star}(T))=\int_0^1\epsilon(\mathcal{K}'(u^{\star}(T)+r\epsilon(\Delta _{\epsilon}(T)+P(T)), \Delta _{\epsilon}(T)+P(T)))dr
$$
which implies
$$
\frac{\mathcal{K}(u_{\epsilon}(T))-\mathcal{K}(u^{\star}(T))}{\epsilon}=\int_0^1(\mathcal{K}'(u^{\star}(T)+r\epsilon(\Delta _{\epsilon}(T)+P(T)), \Delta _{\epsilon}(T)+P(T)))dr.
$$
Using the Lemma ( \ref{estimatLemma}) and the dominated convergence theorem we have
$$
\mathbb{E}(\frac{\mathcal{K}((u_{\epsilon}(T))-\mathcal{K}(u^{\star}(T)))}{\epsilon})\rightarrow  \mathbb{E}(\mathcal{K}'(u^{\star}(T),P(T))), \textmd { as } \epsilon\rightarrow0^+.
$$
In particular, we have
\begin{equation}\label{estimat3}
\mathbb{E}(\mathcal{K}(u_{\epsilon})(T)-\mathcal{K}(u^{\star})(T))=\epsilon\mathbb{E}(\mathcal{K}'(u^{\star}(T),P(T)))+o(\epsilon).
\end{equation}
Now we need to estimate the second factor on the right hand side of (\ref{estimfunct}), first we will observe that
$$\begin{array}{rl}
\mathcal{L}(u_{\epsilon}\!\!\!\!&(t),\Phi_{\epsilon}(t))-\mathcal{L}(u^{\star}(t),\Phi^{\star}(t))=\\
\!\!\!&\mathcal{L}(u_{\epsilon}(t),\Phi_{\epsilon}(t))-\mathcal{L}(u^{\star}(t),\Phi_{\epsilon}(t))+\mathcal{L}(u^{\star}(t),\Phi_{\epsilon}(t))-\mathcal{L}(u^{\star}(t),\Phi^{\star}(t))=\\
&\int_0^1\epsilon((\mathcal{L}_u(u^{\star}(t)+r\epsilon(\Delta _{\epsilon}(t)+P(t)),\Phi^{\star}(t)+\epsilon\Phi(t)), \Delta _{\epsilon}(t)+P(t))dr+\\
&\mathcal{L}(u^{\star}(t),\Phi_{\epsilon}(t))-\mathcal{L}(u^{\star}(t),\Phi^{\star}(t)),
\end{array}
$$
for all $t\in\mathbb{T}.$\\
By applying Lemma  \ref{estimatLemma},  the continuity and boundedness of $\mathcal{L}_u$ in  {\bf H2} and the  dominated convergence theorem we have
$$\begin{array}{rl}
\mathbb{E}\!\!\!\!&(\int_0^T\int_0^1((\mathcal{L}_u(u^{\star}(t)+r\epsilon(\Delta _{\epsilon}(t)+P(t)),\Phi^{\star}(t)+r\epsilon\Phi(t)), \Delta _{\epsilon}(t)+P(t))drdt)\\
&\rightarrow\mathbb{E}(\int_0^T((\mathcal{L}_u(u^{\star}(t),\Phi^{\star}(t)),P(t))dt) \textmd{ as }\epsilon\rightarrow0^+.
\end{array}
$$
Then, 
\begin{equation}\label{esmt2}\begin{array}{rl}
\mathbb{E}&(\int_0^T\mathcal{L}(u_{\epsilon}(s),\Phi_{\epsilon}(s))-\mathcal{L}(u^{\star}(s),\Phi^{\star}(s))ds)=\\
&\epsilon\mathbb{E}(\int_0^T((\mathcal{L}_u(u^{\star}(t),\Phi^{\star}(t)),P(t))dt)+\\
&\mathbb{E}(\int_0^T(\mathcal{L}(u^{\star}(t),\Phi_{\epsilon}(t))-\mathcal{L}(u^{\star}(t),\Phi^{\star}(t)))dt)+o(\epsilon).
\end{array}
\end{equation}
The theorem follows from ( \ref{estimat3}) and \ref{esmt2}).
\cqd\\
The last theorem will play a crucial importance to get the maximum principle. In order to obtain this we need the following Lemmas, the proof follows the ideas of (\cite{AH2}).
\begin{lemma} Suppose that {\bf A1}-{\bf A12}, {\bf H1} and {\bf H2} are satisfied. Then, the solution $(v^{\star},Z^{\star})$ to the equation (\ref{BSE}) satisfies
\begin{equation}\label{inequPrM1}\begin{array}{rl}
\epsilon\mathbb{E}&\left(v^{\star}(T),P(T)\right)+\epsilon\mathbb{E}\left[\int_0^T \mathcal{L}_u(u^{\star}(t),\Phi^{\star}(t))P(t)dt 
\right]+\\
&\mathbb{E}\left[\int_0^T\mathcal{H}(u^{\star}(t),\Phi_{\epsilon}(t),v^{\star}(t),Z^{\star}(t))-\mathcal{H}(u^{\star}(t),\Phi^{\star}(t),v^{\star}(t),Z^{\star}(t)) dt\right]-\\
&\mathbb{E}\left[\int_0^T (v^{\star}(t),B(u^{\star}(t),\Phi_{\epsilon}(t))-B(u^{\star}(t),\Phi^{\star}(t)))dt\right]\geq o(\epsilon).
\end{array}
\end{equation}
\end{lemma} 
\proof  Let us  first observe that $\mathcal{J}(\Phi_{\epsilon})-\mathcal{J}(\Phi^{\star})\geq0$ because $\Phi^{\star}$ is an optimal control. The proof follows from this observation, (\ref{estPrinMax}) and (\ref{hamilt}).
\cqd
\begin{lemma} Suppose that {\bf A1}-{\bf A12} are satisfied. We have
\begin{equation}\label{inequPrM2}\begin{array}{rl}
\mathbb{E}\left(v^{\star}(T),P(T)\right)&=-\mathbb{E}\left[\int_0^T \mathcal{L}_u(u^{\star}(t),\Phi^{\star}(t))P(t)dt 
\right]+\\
&\mathbb{E}\left[\int_0^T (v^{\star}(t),B_\Phi(u^{\star}(t),\Phi^{\star}(t))\Phi(t)dt
\right].
\end{array}
\end{equation}
\end{lemma}
\proof See \cite{AH2} Lemma 6.2.
\cqd
\begin{theorem}Suppose that {\bf A1}-{\bf A12}, {\bf H1} and {\bf H2} are satisfied. Given an optimal pair $(u^{\star},\Phi^{\star}).$ Then, there is a unique pair $(v^{\star},Z^{\star})$ solving the corresponding adjoint BSE (\ref{BSE}) such that the following variational inequality is satisfied
\begin{equation}\label{varioequa}
\left( \nabla_{\Phi}\mathcal{H}(u^{\star}(t),\Phi^{\star}(t),v^{\star}(t),Z^{\star}(t)),\Phi^{\star}-\Phi \right)_{\mathcal{O}}\leq 0,
\end{equation}
for all $\Phi\in \mathcal{O}_1,$ a.e. $t\in\mathbb{T},$ a.s.
\end{theorem}
\proof Given $(u^{\star},\Phi^{\star})$ an optimal pair, thanks to the Theorem \ref{BSExUn}, there is a unique solution $(v^{\star},Z^{\star})$ to (\ref{BSE}).\\
Now, we will demonstrate (\ref{varioequa}). From (\ref{inequPrM1}) and  (\ref{inequPrM2}) we have
\begin{equation}\begin{array}{rl}
\epsilon\mathbb{E}&\left[\int_0^T (v^{\star}(t),B_\Phi(u^{\star}(t),\Phi^{\star}(t))\Phi(t))dt
\right]+\\
&\mathbb{E}\left[\int_0^T\mathcal{H}(u^{\star}(t),\Phi_{\epsilon}(t),v^{\star}(t),Z^{\star}(t))-\mathcal{H}(u^{\star}(t),\Phi^{\star}(t),v^{\star}(t),Z^{\star}(t)) dt\right]-\\
&\mathbb{E}\left[\int_0^T (v^{\star}(t),B(u^{\star}(t),\Phi_{\epsilon}(t))-B(u^{\star}(t),\Phi^{\star}(t)))dt\right]\geq o(\epsilon)
\end{array}
\end{equation}
Furthermore, by using the continuity and boundedness of $B_{\Phi}$ in {\bf A11} and the dominated convergence Theorem, we have
$$\begin{array}{rl}
\frac{1}{\epsilon}\!\!\!\!&\mathbb{E}\left[\int_0^T(v^{\star}(t),-(B(u^{\star}(t),\Phi_{\epsilon}(t))-B(u^{\star}(t),\Phi^{\star}(t)))+\right.\\
&\left.\epsilon B_\Phi(u^{\star}(t),\Phi^{\star}(t)),\Phi(t))dt\right]=\\
&\mathbb{E}\left[\int_0^T(v^{\star}(t),-\int_0^1(B_{\Phi}(u^{\star}(t),\Phi^{\star}(t)+r(\Phi_{\epsilon}(t)-\Phi^{\star}(t)),\Phi(t)))\right.dr+\\
&\left. B_\Phi(u^{\star}(t),\Phi^{\star}(t))(\Phi(t))dt\right]\rightarrow0, \textmd{ as }\epsilon\rightarrow0.
\end{array}
$$
Therefore,
$$
\mathbb{E}\left[\int_0^T\mathcal{H}(u^{\star}(t),\Phi_{\epsilon}(t),v^{\star}(t),Z^{\star}(t))-\mathcal{H}(u^{\star}(t),\Phi^{\star}(t),v^{\star}(t),Z^{\star}(t)) dt\right]\geq o(\epsilon).
$$
Hence, dividing by $\epsilon$ and letting  $\epsilon\rightarrow0$ we obtain
$$
\mathbb{E}\left[\int_0^T(\nabla_{\Phi}\left(u^{\star}(t),\Phi^{\star}(t),v^{\star}(t),Z^{\star}(t)),\Phi(t) \right)_{\mathcal{O}}dt\right]\geq0.
$$
Finally, (\ref{varioequa}) follows by arguing as in (\cite{AH2}).
\cqd

\end{document}